\documentclass[leqno,12pt]{amsart} 
\usepackage{amssymb} 
 
\theoremstyle{plain} 
\newtheorem{thm}{\indent Theorem}[section] 
\newtheorem{lemma}[thm]{\indent Lemma}
\newtheorem{cor}[thm]{\indent Corollary}

\theoremstyle{definition} 

\newtheorem{example}[thm]{\indent Example}

\theoremstyle{remark}

\newcommand{\bC}{{\mathbb C}}

\newcommand{\bR}{{\mathbb R}}
\newcommand{\bZ}{{\mathbb Z}}

\newcommand{\rRe}{{\mathrm {Re}}}
\newcommand{\rIm}{{\mathrm {Im}}}

\begin{document}

\title[SPECIAL LAGRANGIAN SUBMANIFOLDS IN $\bC ^n $ with flat and Fubini-Study metrics]{SPECIAL LAGRANGIAN SUBMANIFOLDS IN $\bC ^n $ with flat and Fubini-Study metrics} 

\author[H. Nakahara]{Hiroshi Nakahara} 

\address{
Department of Mathematics \endgraf 
Tokyo Institute of Technology \endgraf
 2-21-1  O-okayama, Meguro, Tokyo\endgraf
Japan
}
\email{nakahara.h.ab@m.titech.ac.jp}

\maketitle

\begin{abstract}
We give a necessary and sufficient condition for a special Lagrangian submanifold in $\bC ^n$ constructed by Lawlor being also special Lagrangian in $(\bC ^n , \omega_{FS}),$ where $\omega_{FS}$ is the Fubini-Study form.

\end{abstract}

\section{Introduction}
\quad In \cite{Mc}, McLean proved that the moduli space of compact special Lagrangian deformations in Calabi-Yau manifold is a smooth manifold of dimension of the $1^{\rm{th}}$ Betti number $b^1$ of the special Lagrangian submanifold. Gross, Huybrechts and Joyce extended this result to almost Calabi-Yau case in \cite{GHJ}. The moduli space of non-compact special Lagrangian deformations in (almost) Calabi-Yau manifold is also studied. We construct a non-compact special Lagrangian submanifold in almost Calabi-Yau $(\bC ^n , \omega_{FS},dz_1 \wedge \cdots \wedge dz_n).$
Our ambient spaces are always the complex Euclidean space $\mathbb{C}^n$ with coordinates $z_j =x_j +iy_j$ and the standard almost complex structure $J$ with $J(\partial /\partial x_j)=\partial /\partial y_j .$ 
In $\bC ^n,$ we consider the standard symplectic form $\omega _{st}=(i/2)\sum_{j=1}^n dz_j \wedge d\bar{z}_j ,$ the Fubini-Study form $\omega_{FS} =(i/2)\partial \bar{\partial} \log (1+|z_1|^2 +\cdots +|z_n|^2)$ and the holomorphic $(n,0)$-form $\Omega = dz_1 \wedge \cdots \wedge dz_n .$
Then $(\bC ^n , J,\omega_{st}, \Omega)$ is Calabi-Yau and $(\bC ^n , J, \omega_{FS}, \Omega)$ is almost Calabi-Yau. 
A \textit{Lagrangian submanifold} is a real $n$-dimensional submanifold on which the 
symplectic form vanishes. A Lagrangian submanifold $L$ in $(\bC ^n ,J, \omega_{st}, \Omega)$ or $(\bC ^n , J, \omega_{FS}, \Omega)$ is called special Lagrangian if $(\rIm \, \Omega)|_L =0$ holds. It is well known that special Lagrangian submanifolds 
in Calabi-Yau manifolds are area-minimazing. Furthermore, special Lagrangian submanifolds in almost 
Calabi-Yau manifolds are also area-minimazing with respect to a deformed metric $cg,$ where $c$ is some positive function and $g$ is the ${\rm K\ddot{a}hlarian}$ metric on the almost Calabi-Yau manifold. 

By the equation $$\omega_{FS}=\frac{1}{1+|z_1|^2 +\cdots +|z_n|^2}\,\omega_{st}-\frac{i}{2(1+|z_1|^2 +\cdots +|z_n|^2)^2}\sum_{p,q=1}^n \bar{z}_p z_q dz_p \wedge d\bar{z}_q ,$$
we can see that one half of $\omega_{FS}$ is $\omega_{st}.$
So it is natural to try to construct a Lagrangian submainifolds in $(\bC ^n , \omega_{FS})$ from the set of Lagrangian submanifolds in $(\bC ^n , \omega_{st}).$ 

Let  $a_1 ,\ldots ,a_n>0,$ $\psi_1 ,\ldots ,\psi_n \in \mathbb{R}$ be constants.
Define $r_1 ,\ldots ,r_n  :\mathbb{R}\to \mathbb{R}$ by  
$r_j (s)= \sqrt{1/a_j +s^2},$
and  $\phi _1 ,\ldots ,\phi_n :\mathbb{R}\to \mathbb{R}$  by
\begin{equation*}
\phi_j(s) =\psi _j + \int_{0}^{s} \frac{|t|dt}{(1/a_j +t^2)\sqrt{\prod_{k=1}^{n}(1+a_k  t^2)-1}}.
\end{equation*}
It is well know that the submanifold $L$ in $\mathbb{C}^n$ given by
\begin{equation*}
L=\{(x_1r_1(s) e^{i\phi_1(s)},\ldots,x_n r_n(s)e^{i\phi_n(s)});\sum_{j=1}^n  x_j^2 =1, x_j \in\mathbb{R}, s\in\mathbb{R}\} 
\end{equation*}
is a closed embedded special Lagrangian in $(\bC ^n , \omega_{st})$ diffeomorphic to $\mathcal{S}^{n-1} \times \bR.$ This is one
of Lawlor's examples of special Lagrangian submanifolds \cite{La}.
For a treatment of a more general case we refer the reader to \cite[Theorem 1.3]{N}. The following Theorem \ref{mm} is the main theorem of this paper.
\begin{thm}\label{mm}
In the above situation, $L$ is special Lagrangian in $(\bC ^n , \omega_{FS})$ if and only if  $a_1 =\cdots =a_n .$
\end{thm}
We put $g_{FS}$ the Fubini-Study metric.
In the situation of Theorem \ref{mm}, there exists a unique positive function $c$ on $\bC ^n$ such that the mean curvature vector of $L$ with respect to the metric $cg_{FS}$ is equal to zero-vector for any point of $L.$ It is well known that the mean curvature vector is equal to $\Delta F,$ where $F$ is the position vector $F:L\to \bC ^n$ and $\Delta$ is the Laplace operator with respect to the metric $F^{*}(cg_{FS}).$ So we get the following Corollary \ref{c}.
\begin{cor}\label{c}
In the situation of Theorem \ref{mm}, the position vector $F:L\to \bC ^n$ is a solution of the following Laplace's equation 
$$\Delta u =0,$$ 
where $\Delta$ is the Laplace operator with respect to the metric $F^{*}(cg_{FS}).$
\end{cor}

In \cite{Y}, Yamamoto constructs special Lagrangian submanifolds in Calabi-Yau cones over toric Sasaki-Einstein manifolds. 
In particular, explicit Lagrangian submanifols in almost Calabi-Yau cones are obtained by \cite[Theorem 3.4]{Y}.
The Fubini-Study metric $g_{FS}$ is not a cone metric on $\bC ^n.$ 
But by a slightly modified method of \cite[Theorem 3.4]{Y}, we have the following example.
\begin{example}\label{1}
Let  $\lambda_1 ,\ldots ,\lambda_n ,C \in \bR \setminus \{0\}$ be constants. Assume that 
$\lambda_j (\lambda_j +C) >0$ for all $j=1,\ldots ,n.$
Then the submanifold given by
\begin{equation*}
L=\{(x_1 e^{i(\lambda_1 +C)s},\ldots ,x_n e^{i(\lambda_n +C)s});\sum_{j=1}^n \lambda_j x_j ^2 =C , s\in \bR ,\, x_j \in \bR \}
\end{equation*}
is Lagrangian in $(\bC ^n ,\omega_{FS}).$
(By the proof of Lemma \ref{lem} and the above assumption, $L$ is an immersed submanifold.) From the equation \eqref{det} in the proof of Lemma \ref{lem}, $L$ is special Lagrangian if and only if $\sum_j \lambda_j + nC  + \pi/2 =0.$ 
From the first part of Lemma \ref{lem}, $L$ is also Lagrangian in $(\bC ^n ,\omega_{st})$ if and only if $\lambda_1 =\cdots =\lambda_n .$ 
\end{example}
\subsection*{Acknowledgments}
The author wishes to express his thanks to A. Futaki, H. Kasuya, Y. Nitta, M. Shibata, Y. Terashima, Y. Okitsu and H. Yamamoto for useful comments and encouragement. 
\section{Proof of Theorem \ref{mm}}
 In the situation of Theorem \ref{mm}, if $L$ is Lagrangian in $(\bC ^n ,\omega_{FS})$ then $L$ is special Lagrangian in $(\bC ^n ,\omega_{FS}).$ 
So Theorem \ref{mm} is obtained immediately form the following Lemma \ref{lem} which gives the condition for $L$ being Lagrangian in $(\bC ^n ,\omega_{FS}).$ The proof of Lemma \ref{lem} is based on that of Joyce, Lee and Tsui \cite[Theorem A]{JLT} and Yamamoto \cite[Thorem 3.4]{Y}.
\begin{lemma}\label{lem}
Let $I$ be an open interval in $\bR .$ 
Let $\lambda_1 ,\ldots ,\lambda_n ,C \in \bR \setminus \{0\}$ be constants and $\, \omega _j:I\to \bC \setminus \{0\},\,\,j=1,\ldots ,n,$ smooth functions. 
We write $\omega_j =r_j e^{i\phi_j},$ for functions $r_j :I\to (0,\infty)$ and $\phi_j :I\to \bR$ or $\bR / 2\pi \bZ .$ 
We define a submanifold $\Sigma$ in $\bR^n $ by $\Sigma =\{(x_1 ,\ldots ,x_n)\in \bR ^n ;\sum_{j=1}^n \lambda_j x _j ^2 =C \}, $ and a smooth map $\iota :\Sigma \times I \to \bC ^n$ by 
$\iota ((x_1 ,\ldots ,x_n),s)=(x_1 \omega_1 (s),\ldots ,x_n \omega_n (s)).$
Suppose the following mild conditions that 
\begin{equation}\label{assumption}
\sum_{j=1}^n \frac{\lambda_j x_j ^2 \dot{\omega}_j (s)}{\omega _j (s)}\neq 0
\end{equation}
 and $\dot{\phi}_j (s) >0$ for any $s\in I$ and $(x_1 ,\ldots ,x_n)\in \Sigma.$
Then the submanifold $L$ in $\bC ^n$ given by 
\begin{equation*}
L=\iota(\Sigma \times I)=\{(x_1 \omega _1(s),\ldots ,x_n \omega _n(s));\sum_{j=1}^n \lambda_j x_j ^2 =C , s\in I,\, x_j \in \bR \}
\end{equation*}
is an immersed Lagrangian submanifold in  $(\bC ^n , \omega_{st})$ if and only if 
\begin{equation}\label{Im}
\rIm \left(\frac{\dot{\omega}_1 (s) \bar{\omega}_1 (s)}{\lambda_1}\right)=\cdots=\rIm \left(\frac{\dot{\omega}_n (s) \bar{\omega}_n (s)}{\lambda_n}\right)
\end{equation}
holds for any $s\in I.$
Furthermore, $L$ is immersed Lagrangian in both $(\bC ^n , \omega_{st})$ and $(\bC ^n , \omega_{FS})$ if and only if equation \eqref{Im} and
\begin{equation}\label{2}
\frac{|\omega_1 (s)|^2}{\lambda_1}=\cdots=\frac{|\omega_n (s)|^2}{\lambda_n}
\end{equation}
hold for any $s\in I.$
\end{lemma}
\begin{proof}
Fix $x=(x_1 ,\ldots , x_n)\in \Sigma $ and $s\in I .$
Let $e_1 ,\ldots,e_{n-1}$ be an orthonormal basis for $T_x \Sigma$ in $\bR ^n ,$ and write 
$e_j =(a_{j1},\ldots ,a_{jn})$ in $\bR ^n$ for $j=1,\ldots ,n-1.$
Let $e_n =(\sum_{j=1}^{n}\lambda_j ^2 x_j ^2)^{-1/2}\cdot (\lambda_1 x_1,\ldots , \lambda_n x_n).$ Then $e_n$ is unit normal vector to $\Sigma$ at $x$ in $\bR ^n .$ 
Let $e_1 ,\ldots ,e_{n-1}$ be chosen so that $\det (e_1 ,\ldots ,e_{n-1},e_n)=1,$ regarding 
$e_1 ,\ldots ,e_n$ as column vectors, and $(e_1 ,\ldots ,e_n)$ as $n\times n$ matrix.
Now $e_1 ,\ldots ,e_{n-1},\partial /\partial s$ is a basis for $T_{(x,s)}(\Sigma \times I).$
Then we have 
$\iota _* (e_j)=(a_{j1}\omega_1(s),\ldots ,a_{jn}\omega_n (s)),$ for $j<n,$ and 
$\iota_* (\partial /\partial s)=(x_1 \dot{\omega}_1 (s),\ldots ,x_n \dot{\omega}_n (s)).$
Then we compute
\begin{equation}\label{det}
\begin{split}
\det (\iota _* (e_1),\ldots ,\iota _* (e_{n-1}),\iota_* (\partial /\partial s))=&
\begin{vmatrix}
a_{11}\omega _1&\cdots & a_{(n-1)1}\omega _1& x_1 \dot{\omega}_1\\
\vdots & &\vdots &\vdots \\ 
a_{1n}\omega_n &\cdots & a_{(n-1)n}\omega _n& x_n \dot{\omega}_n\\
\end{vmatrix}
\\
=&
\begin{matrix}
\omega_1 \cdots \omega _n 
\end{matrix}
\begin{vmatrix}
a_{11} &\cdots & a_{(n-1)1} & x_1 \dot{\omega}_1 /\omega_1 \\
\vdots & &\vdots &\vdots \\ 
a_{1n} &\cdots & a_{(n-1)n} &  x_n \dot{\omega}_n /\omega_n\\
\end{vmatrix}\\
=& \omega_1 \cdots \omega_n \left< e_n ,\left(\frac{x_1  \dot{\omega}_1}{\omega _1},\ldots \frac{x_n  \dot{\omega}_n}{\omega _n}\right) \right>  \\
=&\omega_1 \cdots \omega_n \sum_{j=1}^n \frac{\lambda_j x_j}{\sqrt{\sum_{k=1}^n \lambda_k ^2 x_k ^2}}\cdot \frac{x_j  \dot{\omega}_j}{\omega _j}\\
=& \frac{\omega_1 \cdots \omega _n }{\sqrt{\sum_{k=1}^n \lambda_k ^2 x_k ^2}}\sum_{j=1}^n \frac{\lambda_j x_j ^2 \dot{\omega}_j}{\omega _j}.\\
\end{split}
\end{equation}
Since we have the assumption \eqref{assumption}, we get $\det (\iota _* (e_1),\ldots ,\iota _* (e_{n-1}),\iota_* (\partial /\partial s))\neq 0.$
This implies that $\iota$ is an immersion.
Put $\omega =\omega_{st} \,\,{\rm or} \,\,\omega_{FS}.$
Define $t:\bC ^n \to \bC ^n$ by $t(z_1 ,\ldots , z_n)=(e^{i\phi _1}z_1 ,\ldots ,e^{i\phi _n}z_n).$
Then we have $t^* \omega =\omega .$
Fix $1\leq j\leq n-1$ and $1\leq k\leq n-1.$ We compute 
\begin{equation*}
\begin{split}
\omega(\iota _* (e_j),\iota _* (e_k)) =& \omega _{(x_1 \omega_1 ,\ldots ,x_n \omega_n)}((a_{j1}\omega_1,\ldots ,a_{jn}\omega_n ),(a_{k1}\omega_1,\ldots ,a_{kn}\omega_n ))\\
 =& \omega _{t(x_1 r_1 ,\ldots ,x_n r_n)}(dt_{(x_1 r_1 ,\ldots ,x_n r_n)}(a_{j1}r_1,\ldots ,a_{jn}r_n )\\
&\quad \quad \quad \quad \quad \quad \quad \quad ,dt_{(x_1 r_1 ,\ldots ,x_n r_n)}(a_{k1}r_1,\ldots ,a_{kn}r_n ))\\ 
=& (t^*\omega) _{(x_1 r_1 ,\ldots ,x_n r_n)}((a_{j1}r_1,\ldots ,a_{jn}r_n ),(a_{k1}r_1,\ldots ,a_{kn}r_n ))\\
=& \omega _{(x_1 r_1 ,\ldots ,x_n r_n)}((a_{j1}r_1,\ldots ,a_{jn}r_n ),(a_{k1}r_1,\ldots ,a_{kn}r_n )).\\
\end{split}
\end{equation*}
By the above calculation and the fact that the real form $\bR ^n \subset \bC ^n$ is a Lagrangian in both $(\bC ^n ,\omega_{st})$ and $(\bC ^n , \omega_{FS}),$ we have $\omega(\iota _* (e_j),\iota _* (e_k)) =0.$ So in order to find the condition $\iota ^* \omega =0,$ we compute that of $\omega(\iota _* (e_j),\iota _* (\partial /\partial s)) =0$ for $1\leq j\leq n-1.$ We obtain  
\begin{equation}\label{T}
\begin{split}
\omega(\iota _* (e_j),\iota _* (\partial /\partial s)) =& \omega _{(x_1 \omega_1 ,\ldots ,x_n \omega_n)}((a_{j1}\omega_1,\ldots ,a_{jn}\omega_n ),(x_1 \dot{\omega}_1,\ldots ,x_n \dot{\omega}_n ))\\
=& \omega _{(x_1 \omega_1 ,\ldots ,x_n \omega_n)}((a_{j1}\omega_1,\ldots ,a_{jn}\omega_n )\\
 &\quad \quad \quad \quad \quad \quad  \quad \quad ,(x_1 (\dot{r}_1 +i r_1 \dot{\phi}_1)e^{i\phi_1},\ldots ,x_n (\dot{r}_n +i r_n \dot{\phi}_n)e^{i\phi_n} ))\\
 =& \omega _{t(x_1 r_1 ,\ldots ,x_n r_n)}(dt_{(x_1 r_1 ,\ldots ,x_n r_n)}(a_{j1}r_1,\ldots ,a_{jn}r_n )\\
 &\quad \quad \quad \quad \quad \quad  \quad \quad , dt_{(x_1 r_1 ,\ldots ,x_n r_n)}(x_1 (\dot{r}_1 +i r_1 \dot{\phi}_1),\ldots ,x_n (\dot{r}_n +i r_n \dot{\phi}_n) ))\\ 
=& (t^*\omega) _{(x_1 r_1 ,\ldots ,x_n r_n)}((a_{j1}r_1,\ldots ,a_{jn}r_n )\\
&\quad \quad \quad \quad \quad \quad  \quad \quad,(x_1 (\dot{r}_1 +i r_1 \dot{\phi}_1),\ldots ,x_n (\dot{r}_n +i r_n \dot{\phi}_n) ))\\
=& \omega _{(x_1 r_1 ,\ldots ,x_n r_n)}((a_{j1}r_1,\ldots ,a_{jn}r_n ),(x_1 (\dot{r}_1 +i r_1 \dot{\phi}_1),\ldots ,x_n (\dot{r}_n +i r_n \dot{\phi}_n) ))\\
=& \omega _{(x_1 r_1 ,\ldots ,x_n r_n)}((a_{j1}r_1,\ldots ,a_{jn}r_n ),J(x_1  r_1 \dot{\phi}_1,\ldots ,x_n  r_n \dot{\phi}_n )).\\
\end{split}
\end{equation}

From the above calculation and $\omega_{st} (\cdot ,J \cdot)= \langle \cdot ,\cdot \rangle,$ we have
\begin{equation*}
\begin{split}
\omega_{st}(\iota _* (e_j),\iota _* (\partial /\partial s)) =&\rRe \left( \sum _{l=1}^n (a_{jl}r_l \cdot \overline{x_l r_l \dot{\phi}_l}) \right)\\
 =&\sum _{l=1}^n a_{jl} x_l r_l ^2 \dot{\phi}_l\\
=&\sum _{l=1}^n a_{jl} x_l \,\rIm (\bar{\omega}_l \dot{\omega}_l )\\
=&\langle  e_j ,(x_1 \rIm (\bar{\omega}_1 \dot{\omega}_1 ),\ldots ,x_n \rIm (\bar{\omega}_n \dot{\omega}_n ) )\rangle \\
=&\langle  e_j ,\left(\lambda_1 x_1 \frac{\rIm (\bar{\omega}_1 \dot{\omega}_1)}{\lambda_1} ,\ldots ,\lambda_n x_n \frac{\rIm (\bar{\omega}_n \dot{\omega}_n ) }{\lambda_n}\right)\rangle .\\ 
\end{split}
\end{equation*}
Thus $\iota ^* \omega _{st} =0$ if and only if the equation \eqref{Im} holds.
This completes the proof of the first part of Lemma \ref{lem}. 

Put $\tau=\sum_{p,q=1}^n \bar{z}_p z_q dz_p \wedge d\bar{z}_q.$ 
In $\bC ^n$ with coordinates $(z_1 ,\ldots ,z_n),$ we have
$$\omega_{FS}=\frac{1}{1+|z_1|^2 +\cdots +|z_n|^2}\,\omega_{st}-\frac{i}{2(1+|z_1|^2 +\cdots +|z_n|^2)^2}\,\tau .$$
Therefore $\iota ^* \omega_{st}=0$ and $\iota ^* \omega_{FS}=0$ if and only if the equation \eqref{Im} and
$\tau(\iota _* (e_j),\iota _* (\partial /\partial s)) =0$ hold for $1\leq j\leq n-1.$
Similarly to the calculation \eqref{T}, we obtain 
$$\tau(\iota _* (e_j),\iota _* (\partial /\partial s)) =\tau _{(x_1 r_1 ,\ldots ,x_n r_n)}((a_{j1}r_1,\ldots ,a_{jn}r_n ), J(x_1  r_1 \dot{\phi}_1,\ldots ,x_n  r_n \dot{\phi}_n )).$$
Therefore we have 
\begin{equation*}
\begin{split}
\quad &\tau(\iota _* (e_j),\iota _* (\partial /\partial s)) \\
=&\sum_{p,q}x_p r_p x_q r_q(a_{jp}r_p (-i) x_q r_q \dot{\phi}_q 
  -a_{jq}r_q i x_p r_p \dot{\phi}_p ) \\
=&-i\sum_{p,q}x_p x_q ^2 r_p ^2 r_q ^2 a_{jp} \dot{\phi}_q -i\sum_{p,q}x_p ^2 x_q r_p ^2 r_q ^2 a_{jq} \dot{\phi}_p\\
=&-i\sum_{p,q}x_p x_q ^2 r_p ^2 r_q ^2 a_{jp} \dot{\phi}_q -i\sum_{p,q}x_q ^2 x_p r_q ^2 r_p ^2 a_{jp} \dot{\phi}_q\\
=&-2i\sum_{p,q}x_p x_q ^2 r_p ^2 r_q ^2 a_{jp} \dot{\phi}_q \\
=&-2i(\sum _p x_p r_p ^2 a_{jp})\cdot(\sum_q x_q ^2 r_q ^2 \dot{\phi}_q) \\
=&-2i\langle (x_1 r_1 ^2 ,\ldots ,x_n r_n ^2) , e_j \rangle \cdot(\sum_q x_q ^2 r_q ^2 \dot{\phi}_q)\\
=&-2i\langle (\lambda_1 x_1 \frac{r_1 ^2}{\lambda_1} ,\ldots ,\lambda_n x_n \frac{r_n ^2}{\lambda_n}) , e_j \rangle \cdot(\sum_q x_q ^2 r_q ^2 \dot{\phi}_q)\\
=&-2i\langle (\lambda_1 x_1 \frac{|\omega_1 | ^2}{\lambda_1} ,\ldots ,\lambda_n x_n \frac{|\omega_n| ^2}{\lambda_n}), e_j  \rangle \cdot(\sum_q x_q ^2 r_q ^2 \dot{\phi}_q).\\
\end{split}
\end{equation*}
From the above calculations and the assumption $\dot{\phi}_j >0,$ we can see that $\iota ^* \omega_{st}=0$ and $\iota ^* \omega_{FS}=0$ if and only if the equations \eqref{Im} and \eqref{2} hold.
This finishes the proof of Lemma \ref{lem}.
\end{proof}
By this Lemma, we obtain Theorem \ref{mm}.

\end{document}